\documentclass{amsart}

\usepackage[margin=1.45in]{geometry}
\usepackage{amsmath,amssymb,amscd} 
\usepackage{graphicx}
\DeclareGraphicsExtensions{.eps} 
\usepackage{mathrsfs}
\usepackage[mathcal]{eucal}

\newtheorem{Thm}{Theorem}{\bfseries}{\itshape}
\newtheorem{Cor}{Corollary}{\bfseries}{\itshape}
\newtheorem{Prop}[Cor]{Proposition}{\bfseries}{\itshape}
\newtheorem{Lem}[Cor]{Lemma}{\bfseries}{\itshape}
{\bfseries}{\itshape}
\newtheorem{Claim}[Cor]{Claim}{\bfseries}{\itshape}
\newtheorem{Fact}[Cor]{Fact}{\bfseries}{\itshape}
\newtheorem{Def}[Cor]{Definition}{\bfseries}{\rmfamily}
\newtheorem{Ex}[Cor]{Example}{\scshape}{\rmfamily}
\newtheorem{Rem}[Cor]{Remark}{\scshape}{\rmfamily}

\renewcommand\ge{\geqslant} \renewcommand\le{\leqslant}
\let\tildeaccent=\~ \let\hataccent=\^
\renewcommand\~[1]{\widetilde{#1}} \renewcommand\^[1]{\widehat{#1}}

\def\<{\left<} \def\>{\right>} \def\({\left(} \def\){\right)}

\def\abs#1{\left\vert #1 \right\vert} 

\let\parasymbol=\S \def\secref#1{\parasymbol\ref{#1}}
 \def\pd#1#2{\frac{\partial#1}{\partial#2}}

\let\polishL=l \def\Zoladek.{\.Zol\c adek}

 \def\const{\operatorname{const}}
\def\codim{\operatorname{codim}}

 \def\dist{\operatorname{dist}}
 \def\ord{\operatorname{ord}}
 \def\etc.{\emph{etc}.}
 
\def\:{\colon} \def\R{{\mathbb R}} \def\C{{\mathbb C}}  \def\N{{\mathbb N}} \def\Q{{\mathbb Q}} 
 \def\e{\varepsilon} \def\S{\varSigma}

 \let\PolishL=\L 
\def\Lojas.{\PolishL ojasiewicz}
\def\cN{{\mathcal N}} \def\cM{{\mathcal M}}

\def\cP{{\mathcal P}}

\def\cT{{\mathcal{T}}}

\def\NI{{\mathcal{NI}}}
\def\NC{{\mathcal{NC}}}
\def\IL{{\mathcal{IL}}}
\def\cC{{\mathcal{cC}}}

 \def\mult{\operatorname{mult}}

\def\rest#1{{\vert_{#1}}}

\def\abs#1{\left\vert #1 \right\vert} \def\norm#1{\left\Vert #1
  \right\Vert}

   \def\CZ{C^Z}

\newcommand{\mo}[1][k]{M^{\smash{(#1)}}}
\newcommand{\bmo}{M}

\begin{document}

\title{Multiplicities of Noetherian deformations} \author{Gal
  Binyamini}\address{University of Toronto} \author{Dmitry Novikov}
\address{Weizmann Institute of Science\\Rehovot\\Israel}

\begin{abstract}
The \emph{Noetherian class} is a wide class of functions 
defined in terms of polynomial partial differential 
equations. It includes functions appearing naturally in 
various branches of mathematics (exponential, elliptic, 
modular, etc.). A conjecture by Khovanskii states that the 
\emph{local} geometry of sets defined using Noetherian 
equations admits effective estimates analogous to the 
effective \emph{global} bounds of algebraic geometry.  

We make a major step in the development of the theory of 
Noetherian functions by 
providing an effective upper bound for the local number of 
isolated solutions of a Noetherian system of equations 
depending on a parameter $\e$, which remains valid even when 
the system degenerates at $\e=0$. An estimate of this sort 
has played the key role in the development of the theory of 
Pfaffian functions, and is expected to lead to similar 
results in the Noetherian setting. We illustrate this by 
deducing from our main result an effective form of the 
\Lojas. inequality for Noetherian functions.
\end{abstract}
\maketitle
\date{\today}

\section{Introduction}
\label{sec:intro}

One of the cornerstones of algebraic geometry is the Bezout theorem: a
system of polynomial equations in a complex projective space always
admits a specified number of solutions depending on their degrees.
This statement has profound implications for the algebraic category:
essentially every geometric and topological property of an
algebraic variety can be estimated in terms of the degrees of the
equations defining it.

Moving beyond the algebraic category, Khovanskii has defined the class
of real \emph{Pfaffian functions}. In his theory of Fewnomials 
\cite{Khovanskii:Fewnomials},
Khovanskii has shown that the number of solutions of a system of real
Pfaffian equations admits an effective upper bound in terms of the
degrees of the equations. This fundamental result has been the basis
of many subsequent works, showing that the geometry of real Pfaffian
sets is tame and admits effective estimates in terms of degrees 
\cite{W:expmod,S:pfaffclos,GV:PfaffianComplexity}.

The real Pfaffian class consists of functions satisfying a system of
differential equations with a certain triangularity condition. It is
surprisingly general, and includes many important transcendental
functions --- most notably the real exponential. On the other hand,
not all differential systems appearing naturally in mathematics are
real Pfaffian. We mention a few key examples:
\begin{itemize}
\item Exponential maps of (complex) commutative algebraic groups, for
  instance complex exponentials and elliptic functions. Solutions of
  equations involving such functions have been studied in the context
  of diophantine approximation, \cite{MW:zero1,MW:zero2,W:multgrps}.
\item Abelian integrals and iterated Abelian integrals. Solutions of
  equations involving such functions have been studied in relations to
  perturbations of Hamiltonian systems and their limit cycles 
  \cite{BN:iterated,BNY:inf16}.
\item Functions of modular type, for instance Klein's modular invariant
  $j$ and Ramanujan's functions $P,Q,R$. Solutions of equations
  involving such functions have been studied in transcendental number
  theory, \cite{N:PQR,Philibert:MahlerManin}.
\item Hamiltonian flow maps in completely integrable systems.  
\end{itemize}

At the time of the development of the theory of Fewnomials, Khovanskii
also considered the more general notion of \emph{Noetherian
  functions}, excluding the triangularity condition from the
definition of the Pfaffian functions. One may loosely say that a
collection of functions is Noetherian if each of their derivatives can
be expressed algebraically in terms of the functions themselves
(see~\secref{ssec:noeth} for two equivalent precise definitions). All
functions listed above, while not real-Pfaffian, do form Noetherian
systems.

Noetherian functions do not satisfy global estimates similar to those
obtained in the theory of Fewnomials. However, Khovanskii has
conjectured that a local analog of these estimates continues to hold
in the class of Noetherian functions, which would form a type of local
analog of the Pfaffian class. As in the Pfaffian case, the key step is
to establish an upper bound for the number of solutions of a system of
equations involving Noetherian functions in terms of degrees --- this
time in a suitably defined local sense.

In this paper we make a main step in the development of the theory of
Noetherian functions by establishing an upper bound for the local
number of solutions of a system of Noetherian equations depending on a
Noetherian parameter $\e$. The novelty of our result is that the
estimate is valid even if the system degenerates and has non-isolated
solutions when $\e=0$. A result of this type has been established by
Gabrielov in the (complex) Pfaffian case, and subsequently used to
derive estimates for many topological and geometric properties in the
Pfaffian category (see \cite{GV:PfaffianComplexity} for a survey). As an
illustrative example, we use our estimate to establish an effective
\Lojas. inequality in the Noetherian category following Gabrielov. We
expect many other results of \cite{GV:PfaffianComplexity} to follow
similarly.

\subsection{Noetherian functions}\label{ssec:noeth}
We begin by defining our principal object of investigation, namely the
rings of Noetherian functions.
\begin{Def}[\protect{Noetherian functions, algebraic \cite{Tougeron:Noetherian}}]
  A ring $S$ of analytic functions in a domain $U\subset\C^n$ is
  called a \emph{ring of Noetherian functions} if it is generated over
  the polynomial ring $\C_n$ by functions $\phi_1,\ldots,\phi_m\in S$
  and closed under differentiation with respect to each variable
  $x_i$.

  Any element of $\psi\in S$ is called a \emph{Noetherian
    function}. We will say that it has degree $d$ with respect to
  $\phi_1,\ldots,\phi_m$ if $d$ is the minimal degree of a polynomial
  $P\in\C_{n+m}$ such that
  $\psi=P(x_1,\ldots,x_n,\phi_1,\ldots,\phi_m)$. When there is no risk
  of confusion we will simply call this the degree of $\psi$.
\end{Def}

Alternatively, equip $\C^{n+m}$ with the affine coordinates
$(x_1,\ldots,x_n,f_1,\ldots,f_m)$ and denote by $\C_{n+m}$ the ring of polynomials
in these variables. Consider a distribution $\Xi\subset T\C^{n+m}$
defined by
\begin{equation} \label{eq:noetherian-system}
  \Xi=\<V_1,\ldots,V_n\>, \quad V_i = \pd{}{x_i}+\sum_{j=1}^m g_{ij}\pd{}{f_j} \quad \text{for } i=1,\ldots,n.
\end{equation}
where $g_{ij}\in\C_{n+m}$. We call $\delta=\max_{ij}\deg g_{ij}$ the
\emph{degree of the chain}. For any point
where~\eqref{eq:noetherian-system} is integrable, we will denote by
$\Lambda_p$ the germ of an integral manifold through $p$.

\begin{Def}[\protect{Noetherian functions, geometric \cite{GK:MultNoetherian}}] \label{def:NoethFgeom}
  A tuple $\phi_1,\ldots,\phi_m$ of analytic functions on a domain
  $U\subset\C^n$ is called a \emph{Noetherian chain} if its graph
  forms an integral manifold of the distribution $\Xi$ (for some
  choice of the coefficients $g_{ij}\in\C_{n+m}$). In other words,
  if the following system of differential equations is satisfied
  \begin{equation}
    \pd{\phi_j}{x_i} = g_{ij}(x_1,\ldots,x_n,\phi_1,\ldots,\phi_m).
  \end{equation}

  Any function $\psi=P(x_1,\ldots,x_n,\phi_1,\ldots,\phi_m)$ where
  $P\in\C_{n+m}$ is called a Noetherian function. We will say that it
  has degree $d$ with respect to $\phi_1,\ldots,\phi_m$ if $d$ is the
  minimal degree of such a polynomial $P\in\C_{n+m}$. When there is no
  risk of confusion we will simply call this the degree of $\psi$.
\end{Def}
It is straightforward to check that the generators of a ring of
Noetherian functions forms a Noetherian chain and vice versa.

\begin{Rem} \label{rem:integral-manifold} Note that by the above, if
  $\Lambda$ is an integral manifold of~\eqref{eq:noetherian-system}
  then the projection onto the $x$-variables gives a system of
  coordinates on $\Lambda$, and with respect to these coordinates the
  Noetherian functions are simply the restrictions of polynomials
  $P\in\C_{n+m}$ to $\Lambda$. This viewpoint is particularly
  convenient for our purposes and will be used frequently.
\end{Rem}

We define a \emph{Noetherian set} to be the common zero locus
of any collection of Noetherian functions with a common domain
of definition $U$.

As solutions of non-linear differential equations, Noetherian
functions do not admit good global behavior --- in fact, their domains
of definition may include movable singularities and natural
boundaries. It is natural therefore to begin the study of these
functions with their local properties: what can be said about germs of
Noetherian functions and sets in terms of the
system~\eqref{eq:noetherian-system} that defines them?

In order to be more concrete, consider any
system~\eqref{eq:noetherian-system} with the parameters $m,n,\delta$
and an $n$-tuple $\psi_1,\ldots,\psi_n$ of Noetherian functions of degrees
bounded by $d$. Let $\cN(m,n,\delta,d;0)$ denote the maximal possible
multiplicity of a common zero of the equations $\psi_1=\cdots=\psi_n=0$,
assuming that the zero is isolated.  More generally, consider any
family $\psi^\e_1,\ldots,\psi^\e_n$ of Noetherian functions of degrees
bounded by $d$ (for each fixed $\e$) and depending analytically on
$\e$. Let $\cN(m,n,\delta,d)$ denote the maximal possible number of
isolated zeros (counted with multiplicities) of the equations
$\psi^\e_1=\cdots=\psi^\e_n=0$ which converge to a given point point as
$\e\to0$. Crucially, here it is not assumed that the given point is an
isolated solution of the limiting system at $\e=0$.

It is a general principle that estimates for $\cN(m,n,\delta,d)$ imply
estimates for the local topological and analytic structure of
functions and sets, by a combination of topological arguments (e.g.
Morse theory) and analytic-geometric arguments (e.g. the study of
polar curves). The more restrictive $\cN(m,n,\delta,d;0)$ does not
control the topological and analytic-geometric structure to the same
extent, but it is a convenient first approximation which is often
easier to study and still provides useful information. Instances of
the problem of estimating $\cN(m,n,\delta,d;0)$ and
$\cN(m,n,\delta,d)$ have been studied by authors in various areas of
mathematics. Below we present an outline of some of the main
contributions.

\subsection{Historical sketch}

The case $n=1$ is special and has attracted considerable attention.
In this case the problem reduces to the study of polynomial functions
on the trajectory of a (non-singular) polynomial vector field at a
point. Moreover, by an argument due to Khovanskii
\cite{Gab:MultVectorFieldNew}, in this case the quantities
$\cN(m,n,\delta,d;0)$ and $\cN(m,n,\delta,d)$ coincide. Therefore the
problem is to estimate the multiplicity of the restriction of a
polynomial to the trajectory of a non-singular vector field at a
point.

Brownawell and Masser \cite{BM:MultEstI,BM:MultEstII} and Nesterenko
\cite{Nesterenko:MultEstimate} have studied the problem of estimating
$\cN(m,1,\delta,d)$ with motivations in transcendental number theory.
These authors have also considered the case where one takes into
account the multiplicity at several (rather than just one) points.
Nesterenko also considered singular systems of differential equations
in \cite{N:PQR} and established an estimate which is sharp with
respect to $d$ up to a multiplicative constant and doubly-exponential
with respect to $m$. This estimate was improved to single-exponential
in $m$ in \cite{B:mult-sing}.

Risler \cite{Risler:Nonholonomy,GR:MultVectorFields} has considered
the problem in the context of control theory and specifically
non-holonomic systems. Gabrielov \cite{Gab:MultVectorFieldNew} refined
this work and found a surprising reformulation in terms of the Milnor
fibers of certain deformations, leading to an estimate for
$\cN(m,1,\delta,d)$ which is simply-exponential in $m$ and polynomial
in $d$ and $\delta$. Gabrielov's approach was refined in
\cite{B:mult-morse} to give a sharp asymptotic with respect to $d$.

Novikov and Yakovenko \cite{NY:Meandering} have considered the problem
with motivations in the study of abelian integrals, and have produced
a bound which is valid not only for local multiplicity but for the
number of zeros (counted with multiplicities) in a small ball of a
specified radius. Yomdin \cite{Yomdin:Oscillation} has obtained a
similar result motivated by the study of cyclicities in dynamical
systems. For both of these works, the principal difficulty is to
obtain an estimate in an interval whose radius does not degenerate
along deformations where the polynomial becomes identically vanishing
on the trajectory.

The case $n>1$ is considerably more involved, because germs of
functions and sets in several variables can degenerate in much more
complicated ways. Khovanskii \cite{Khovanskii:Fewnomials} has studied
systems~\eqref{eq:noetherian-system} satisfying an additional
triangularity condition and defined over the reals, leading to the
influential theory of Pfaffian functions. Khovanskii has shown that
the global geometry of Pfaffian sets (e.g. number of connected
components, sum of Betti numbers) can be estimated in terms of the
degrees of the coefficients of the
system~\eqref{eq:noetherian-system}.

Gabrielov \cite{Gab:MultPfaffian} has established a local complex
analog of Khovanskii's estimates for the number of solutions of a
system of Pfaffian equations, namely an upper bound for
$\cN(m,n,\delta,d)$ (for the Pfaffian case). Gabrielov used this
result to produce an effective form of the \Lojas. inequality for
Pfaffian functions, and later in a series of joint works with Vorobjov
(for example see
\cite{GV:PfaffianComplexity,GV:CylindricalDecomp,GV:PfaffianStrats})
established many results on the complexity of various geometric
constructions within the Pfaffian class (e.g. stratification, cellular
decomposition, closure and frontier).

In the early 1980s Khovanskii conjectured that the quantity
$\cN(m,n,\delta,d)$ admits an effective upper bound (see
\cite{GK:MultNoetherian} for the history and another equivalent
form of this conjecture). The conjecture in this generality has
remained unsolved.

Building upon the one-dimensional approach of Gabrielov
\cite{Gab:MultVectorFieldNew}, Gabrielov and Khovanskii
\cite{GK:MultNoetherian} have established an estimate for
$\cN(m,n,\delta,d;0)$ using a topological deformation technique. In
our previous work \cite{BN:MultOp} we have established a weaker
estimate for this quantity using algebraic techniques, and extended
these local estimates to an estimate on the number of zeros (counted
with multiplicities) in a sufficiently small ball, provided that the
ball is not too close to a non-isolated solution of the
system. However, neither of these approaches give an estimate for the
more delicate quantity $\cN(m,n,\delta,d)$.

In \cite{BN:NoetherianDim2} we began the investigation of non-isolated
solutions of systems of Noetherian equations for $n=2$. We have
obtained an explicit upper bound for an appropriately defined
``non-isolated intersection multiplicity''. However, this bound is not
sufficient if one is interested in estimating the number of solutions
born from a deformation of a non-isolated solution.

\subsection{Statement of our result}

Consider a system~\eqref{eq:noetherian-system} and the corresponding
ring of Noetherian functions $S$ defined in a domain $U\subset\C^n$,
and let $p\in U$.

Let $\rho\in S$ and let $X\subset U$ be a germ of a Noetherian set
\begin{equation}
  X = \{x\in U: \psi_1(x)=\cdots=\psi_{n-1}(x)=0 \}
\end{equation}
at the point $p$, where $\psi_i\in S$ for $i=1,\ldots,k$.

\begin{Def}
  The \emph{deformation multiplicity}, or \emph{deflicity} of $X$
  with respect to $\rho$ is the number of isolated points in
  $\rho^{-1}(y)\cap X$ (counted with multiplicities) which converge to $p$
  as $y\to\rho(p)$.
\end{Def}

We let $\cM(m,n,\delta,d)$ denote the maximum possible deformation
multiplicity for any system~\eqref{eq:noetherian-system} with parameters
$m,n,\delta$ and any Noetherian set defined as above with $\deg \psi_i\le d$
and $\deg\rho\le d$.

\begin{Rem}
  This notion is closely related to the quantity
  $\cN(m,n,\delta,d)$ defined in section~\ref{ssec:noeth}. Indeed, let a family
  $\psi^\e_1,\ldots,\psi^\e_n$ be given where $\psi^\e_i\in S$ and
  $\deg \psi^\e_i\le d$. Let us assume further that the dependence on
  $\e$ is Noetherian, i.e. that these functions may be viewed as
  Noetherian functions in the variables $x_1,\ldots,x_n,\e$ and have
  degrees bounded by $d$ in this sense as well. Then the number of
  isolated zeros of $\psi^\e_1=\cdots=\psi^\e_n=0$ which converge to $p$ as
  $\e\to0$ coincides with the deflicity of the set
  \begin{equation}
    X = \{(x,\e)\in U : \psi_1(x,\e)=\cdots=\psi_n(x,\e)=0\}
  \end{equation}
  with respect to $\rho\equiv\e$.
\end{Rem}

We can now state our main result.

\begin{Thm}\label{thm:main}
  The quantity $\cM(m,n,\delta,d)$ admits an effective upper bound,
  \begin{equation}
    \cM(m,n,\delta,d) \le 
    \left(\max\{d,\delta\}(m+n)\right)^{16(m+n)^{20n+3}}=
    \left(\max\{d,\delta\}(m+n)\right)^{(m+n)^{O(n)}}
  \end{equation}
\end{Thm}

Recall that according to Khovanskii's conjecture, the quantity
$\cN(m,n,\delta,d)$ admits an effective upper bound. Our result
applies only under the additional condition that the dependence of the
deformation on $\e$ is Noetherian (the reader may consult
Remark~\ref{rem:why-not-gk} for a discussion of the point in our
argument where this extra condition is needed). However, in most
applications one needs to deal only with explicitly presented
deformations which do depend in a Noetherian manner on $\e$. One can
therefore expect that Theorem~\ref{thm:main} will suffice for the
investigation of many local topological and analytic-geometric
properties of Noetherian sets and
functions.

\begin{Rem}
  One can also consider systems~\eqref{eq:noetherian-system} involving
  rational (rather than polynomial) coefficients, as well as rational $P,R$.
  As long as one considers points $p$ away from the polar locus of these
  functions, Theorem~\ref{thm:main} remains valid and can be proved in the
  same way. On the other hand, when $p$ belongs to the singular locus of the
  system, the situation is considerably more involved.
\end{Rem}

\subsection{Application:  \Lojas. inequality}

As an example of an application of Theorem~\ref{thm:main}, we prove an
effective version of {\L}ojasiewicz inequalities for Noetherian
functions. Let $\Xi_{\R}\subset T\R^{n+m}$ be a real distribution
spanned by real vector fields $V_i$ as in
\eqref{eq:noetherian-system}, with $g_{ij}\in\R_{n+m}$ of degree
$\le\delta$, where $\R_{n+m}=\R[x_1,...,x_n,f_1,...,f_m]$, and let
$\Lambda_p\subset\R^{n+m}$ be a germ of its integral manifold through
$p$. The restrictions of real polynomials of degree $d$ to $\Lambda_p$
are called real Noetherian functions of degree $d$.

\begin{Thm}\label{thm:Lojasiewicz}
  Let $f,g$ be real Noetherian functions of $n$ variables of degree
  $d$ defined by the same Noetherian chain, $f(p)=0$, and assume that
  $\{f=0\}\subset\{g=0\}$ near $p$. Then there exists a constant
  $0<k\le \left(\max\{d,\delta\}(m+n)\right)^{(m+n)^{O(n)}}$ such that
  $|f|>|g|^k$ near $p$.
\end{Thm}

The proof below is essentially the proof from \cite{Gab:MultPfaffian},
with Theorem~\ref{thm:main} replacing \cite[Theorem
2.1]{Gab:MultPfaffian}. Denote the complexifications of $f,g$ by the
same letters. They are evidently Noetherian functions, defined by the
(complexification of the) same Noetherian chain. Denote by
$\Delta\subset\C^2$ the polar curve of $f$ with respect to $g$, i.e.
the set of critical values of the mapping $(f,g):\Lambda_p\to\C^2$.

\begin{Lem}
Let $f=\sum c_i g^{\lambda_i}$ be the Puiseux expansion of an irreducible 
component $\Delta'\not=\{f=0\}$ of $\Delta$. Let $s$ be the least common 
denominator of $\lambda_i$, and let  $r=s\lambda_1$. Then 
$$
r,s\le \left(2(d+\delta-1)(m+n)\right)^{16(m+n)^{20n+3}}
$$
\end{Lem}
\begin{proof}
  Let $\Sigma'\subset\Lambda_p$ be an irreducible component of the
  critical set of the mapping $(f,g)$ whose image is $\Delta'$. Then
  $r$ is at most the number of points in
  $\{f(x)=\epsilon, x\in\Sigma'\}$ converging to $p$ as
  $\epsilon\to 0$, and $s$ is at most the number of points in
  $\{g(x)=\epsilon, x\in\Sigma'\}$ converging to $p$ as
  $\epsilon\to 0$. We can assume that these points are isolated: after
  cutting by a linear space $L'$ of suitable dimension, the polar
  curve of the restriction of $f$ to $L'$ with respect to the
  restriction of $g$ to $L'$ will have $\Delta'$ as its irreducible
  component and the problem is reduced to a similar one in smaller
  dimension. Therefore it is enough to bound from above the deflicity
  of $X=\{dg\wedge df=0\}$ with respect to $f$ or $g$. But $X$ is a
  Noetherian set of degree $2(d+\delta-1)$, so Theorem~\ref{thm:main}
  implies the claim of the Lemma.
\end{proof}

\begin{proof}[Proof of Theorem~\ref{thm:Lojasiewicz}]
  Let $\Sigma$ be the union of points of minima of restrictions of
  $|f|$ to the intersection of level curves of $|g|$ with a small ball
  around $p$. We can assume that the closure of $\Sigma$ contains $p$,
  otherwise the problem is reduced to a similar one of lower
  dimension. Therefore the image of $\Sigma$ under $(f,g)$ belongs to
  a polar curve of $f$ relative to $g$, and not to $\{f=0\}$ by the
  condition. Therefore Theorem~\ref{thm:Lojasiewicz} follows from the
  previous Lemma.
\end{proof}

\subsection{Acknowledgments} We would like to express our gratitude to
A. Gabrielov, A. Khovanskii, P. Milman and S. Yakovenko for valuable
discussions and to the anonymous referee for numerous corrections and
helpful remarks.

\section{Background}

\subsection{Noetherian functions: integrability and multiplicity estimates}

A system of the form~\eqref{eq:noetherian-system} does not necessarily
admit an integral manifold through every point of $\C^{n+m}$, since we
do not assume that the vector fields $V_i$ commute globally. The
\emph{integrability locus} $\IL\subset\C^{n+m}$
of~\eqref{eq:noetherian-system} is defined to be the union of all the
integral manifolds of the system. In \cite{GK:MultNoetherian} it was
shown that $\IL$ is algebraic for any
system~\eqref{eq:noetherian-system}, and moreover that the following
theorem holds.

\begin{Thm}[\protect{\cite[Theorem~4]{GK:MultNoetherian}}] \label{thm:gk-integral}
  Let the system~\eqref{eq:noetherian-system} have parameters $m,n,\delta$.
  Then $\IL$ can be defined as the zero locus of a set of polynomials of
  degrees not exceeding
  \begin{equation}\label{eq:gk-integral}
    d_\IL = \frac{(m+1)(\delta-1)}{2}[2\delta(n+m+2)-2m-2]^{2m+2}+\delta(n+2)-1
  \end{equation}
\end{Thm}

We now state the main result of \cite{GK:MultNoetherian}, namely an
upper bound for the multiplicity of an isolated common zero of a tuple
of Noetherian functions.

\begin{Thm}[\protect{\cite[Theorem~1]{GK:MultNoetherian}}] \label{thm:gk-mult}
  Let the system~\eqref{eq:noetherian-system} have parameters $m,n,\delta$
  and let $\psi_1,\ldots,\psi_n$ be Noetherian functions of degrees at most $d$
  with respect to this system. Then the multiplicity of any isolated solution
  of the equations $\psi_1=\cdots=\psi_n=0$ does not exceed the maximum of
  the following two numbers:
  \begin{eqnarray*}
    &\frac{1}{2}Q\((m+1)(\delta-1)[2\delta(n+m+2)-2m-2]^{2m+2}
      +2\delta(n+2)-2\)^{2(m+n)} \\
    &\frac{1}{2}Q\(2(Q+n)^n(d+Q(\delta-1))\)^{2(m+n)},
    \quad \text{where} \quad Q=e\ n\(\frac{e(n+m)}{\sqrt{n}}\)^{\ln n+1}
      \(\frac{n}{e^2}\)^n
  \end{eqnarray*}
\end{Thm}

\subsection{Multiplicity operators}
\label{sec:multiplicity operators}

We shall make extensive use of the notion of multiplicity operators
introduced in~\cite{BN:MultOp}. The multiplicity operators are a
collection of partial differential operators which can be used to
study the (local structure of) solutions of $n$ equations in $n$
variables. They may be thought of as playing a role similar to the
usual derivatives in the case of functions of a single variable. For
the convenience of the reader we recall in this section the basic
definitions and properties of multiplicity operators that shall be
used in the sequel.

Let $F=(F_1,\ldots,F_n):B^n(0,1)\to\C^n$ be a holomorphic mapping
extendable to a neighborhood of the unit ball $B^n(0,1)$. We will
denote by $\norm{\cdot}$ the maximum norm on the unit ball. As a
notational convenience, all holomorphic functions considered in this
subsection are assumed to be holomorphic in the unit disc.

For every $k\in\N$ there is a finite family of polynomial
differential operators $\{\bmo_B^\alpha\}$ called the \emph{basic
  multiplicity operators}. Any operator $\mo$ in the convex hull of
this set is called a \emph{multiplicity operator} of order $k$. Every
multiplicity operator of order $k$ is a polynomial differential
operator of order $k$ and degree $\dim J_{n,k}-k$, where $J_{n,k}$
denotes the space of $k$-jets of functions in $n$ variables.
We recall that $\dim J_{n,k}=\binom{n+k}k$.

\begin{Ex}
  In the familiar case $n=1$ the basic multiplicity operators of order
  $k$ are given (up to normalization by some universal constants) by
  the first $k$ derivatives. We will outline proofs for the various
  results presented below in this more familiar special case to help
  the reader appreciate the general statements.
\end{Ex}

We will use the convention that when applying a multiplicity
operator $\mo$ to a polynomial $P\in\C_{n+m}$, the derivatives
of the variables $f_1,\ldots,f_m$ with respect to $x_1,\ldots,x_n$
shall be given by the Noetherian system~\eqref{eq:noetherian-system}.
In this way $\mo(P)\in\C_{n+m}$ and moreover, if $\deg P=d$ then
$\deg\mo(P)\le d_M(n,\delta,d,k)$, where
\begin{equation}\label{eq:deg d_M}
  d_M(n,\delta,d,k) = {n+k \choose k}(d+k\delta)-k  
\end{equation}
At points where~\eqref{eq:noetherian-system} is integrable, this
agrees of course with applying $\mo$ to $P$ viewed as a Noetherian
function of $x_1,\ldots,x_n$ by restriction of $P$ to the integral manifold
of~\eqref{eq:noetherian-system} through the point. At points
where~\eqref{eq:noetherian-system} is not integrable $\mo(P)$ has no
intrinsic meaning and will in general depend on the order in which we
choose to evaluate repeated derivatives. We fix this order
arbitrarily. In fact we will only consider $\mo(P)$ on points where
the system is integrable, so this arbitrary choice will make no
difference in our arguments.

We will sometimes need to evaluate multiplicity operators with respect
to a linear subspace of the $x$-coordinates. If $\cT$ is such a linear
subspace and $\mo$ a multiplicity operator of dimension $\dim\cT$, we
will denote by $\mo_\cT$ the operator given by
\begin{equation}
  [\mo_\cT(F)](p) = [\mo(F\rest{\cT_p})](p), \qquad \cT_p:=p+\cT,
\end{equation}
viewed again as an operator acting on $\C_{n+m}$. Also, to simplify
the notation we denote by $\mo_p$ (resp. $\mo_{\cT,p}$) the
differential functional obtained by applying $\mo$ (resp. $\mo_\cT$),
followed by evaluation at the point $p$.

The following proposition is the basic property of the multiplicity
operators.
\begin{Prop}[\protect{\cite[Proposition~5]{BN:MultOp}}] \label{prop:mo-fundamental}
  We have $\mult F_p>k$ if and only if $\mo_p(F)=0$ for all multiplicity
  operators of order $k$.
\end{Prop}
In the case $n=1$, Proposition~\ref{prop:mo-fundamental} corresponds
to the familiar fact that $p$ is a zero of multiplicity $k$ of $F$ if
and only if $F(p)=F'(p)=\cdots=F^{(k)}(p)=0$.

We now state a result relating the multiplicity operators of $F$ to
its set of common zeros.

\begin{Thm}[\protect{\cite[Theorem~1]{BN:MultOp}}] \label{thm:mo-polydisc-zeros}
  Assume that $\norm{F}\le1$ and that $F$ has $k+1$ zeros (counted
  with multiplicities) in the polydisc $D_r^n$. Then for every $k$-th
  multiplicity operator $\mo$,
  \begin{equation}
    \CZ_{n,k} \abs{\mo_0(F))} < r 
  \end{equation}
  where $\CZ_{n,k}$ is a universal constant.
\end{Thm}
We outline the proof of Theorem~\ref{thm:mo-polydisc-zeros} in the
case $n=1$. We assume for simplicity that $r<1/2$. Then by standard
arguments $F$ can be decomposed in the form
\begin{equation}
  F(z) = (z-z_0)\cdots(z-z_k)U(z)
\end{equation}
where $\norm{U}\le 2^{k+1}$ and $z_0,\ldots,z_k$ are the $k+1$ zeros
of $F$ in the disc of radius $r$. It is now straightforward to verify
that the first $k$ derivatives of $F$ are smaller than
$(\CZ_{1,k})^{-1}r$ for an appropriate constant $\CZ_{1,k}$.

We now state some results relating the multiplicity operators of $F$
to its growth around the origin and its behavior under perturbation.

\begin{Thm}[\protect{\cite[Theorem~2]{BN:MultOp}}] \label{thm:mo-sphere-growth}
  Let $\norm{F}\le1$ and $\mo$ a $k$-th multiplicity operator. There
  exist positive universal constants $A_{n,k},B_{n,k}$ with the
  following property:

  For every $r<\abs{\mo_0(F)}$ there exists $A_{n,k} r < \tilde r < r$ such that
  \begin{equation}\label{eq:mo-sphere-growth}
    \norm{F(z)} \ge B_{n,k} \abs{\mo_0(F)} \tilde r^k \mbox{ for every } \norm{z}=\tilde r.
  \end{equation}
\end{Thm}
We outline the proof of Theorem~\ref{thm:mo-sphere-growth} in the
case $n=1$. If $P_k(z)$ denotes the $k$-th Taylor polynomial of
$F$ then
\begin{equation}\label{eq:F-taylor-decomp}
  F(z)=P_k(z)+O(z^{k+1}), \qquad \norm{P_k}=\Omega(\abs{\mo_0(F)}).
\end{equation}
Using polynomial inequalities, for instance the Cartan lemma, one can
show that $|P_k(z)|=\Omega(\abs{\mo_0(F)} z^k)$ for $z$ outside a set of small
discs around the zeros of $P_k$. Subsequently for $|z|$ sufficiently
smaller than $\abs{\mo_0(F)}$ we see that $\abs{P_k(z)}$ dominates
$O(z^{k+1})$, allowing one to derive~\eqref{eq:mo-sphere-growth}
from~\eqref{eq:F-taylor-decomp}.

\begin{Cor}[\protect{\cite[Corollary~15]{BN:MultOp}}] \label{cor:mo-power-pert}
  Let $\norm{F},\norm{G}\le1$ and $\mo$ be a $k$-th multiplicity
  operator. There exist universal constants $A''_{n,k},B''_{n,k}$ with
  the following property:

  For every $\norm{G(0)}<r<B''_{n,k} \abs{\mo_0(F)}$ there exists
  $A''_{n,k} r < \tilde r < r$ such that
  \begin{equation}
    \norm{F(z)} > \norm{G^{k+1}(z)} \text{ for every } \norm{z}=\tilde r
  \end{equation}
  where $G^{k+1}$ is taken component-wise. 
\end{Cor}

In particular, Corollary~\ref{cor:mo-power-pert} in combination
with the Rouch\'e principle implies that if $F,G$ are as in the corollary
and $E$ is a holomorphic function satisfying $\norm{E}<1$ then
\begin{equation}\label{eq:mo-power-rouche}
  \#\{z:F(z)=0, \norm{z}<\tilde r\} = \#\{z:F(z)+E(z) G^{k+1}(z)=0, \norm{z}<\tilde r\}.
\end{equation}

Finally, we state a result relating to multiplicity operators of $F$ to
its rate of growth along analytic curves.

\begin{Thm}[\protect{\cite[Theorem~3]{BN:MultOp}}] \label{thm:mo-curve-growth}
  Let $\mo$ be a multiplicity operator of order $k$ and
  $\gamma\subset(\C^n,0)$ a germ of an analytic curve. Then
  \begin{equation}
    \ord_\gamma \mo(F) \ge \min\{\ord_\gamma f_i: i=1,\ldots,n\} - k
  \end{equation}
\end{Thm}

In the case $n=1$, Theorem~\ref{thm:mo-curve-growth} corresponds to
the familiar fact that taking a $k$-th derivative of a function can
decrease the order of its zero by at most $k$.

\section{Structure of the proof}
\label{sec:proof-structure}

Consider a system~\eqref{eq:noetherian-system} with parameters
$n,m,\delta$ and the corresponding ring of Noetherian functions $S$.
By Remark~\ref{rem:integral-manifold} we view $S$ as the restriction
of the ring $\C_{n+m}$ to an integral manifold $\Lambda_p$
of~\eqref{eq:noetherian-system}.

Let $\rho\in S$ with $\deg\rho\le d$ and let $X\subset\Lambda_p$ be a germ of
a Noetherian set
\begin{equation}
  X = \{x\in\Lambda: \psi_1(x)=\cdots=\psi_{n-1}(x)=0 \}
\end{equation}
at the point $p$, where $\psi_i\in S$ and $\deg\psi_i\le d$ for $i=1,\ldots,n-1$.

By definition, $\rho=R\rest{\Lambda_p}$ and
$\psi_i=P_i\rest{\Lambda_p}$ where $R,P_1,\ldots,P_{n-1}$ are elements
of $\C_{n+m}$ of degrees bounded by $d$. To simplify the notation we
will let $P$ denote the tuple $P_1,\ldots,P_{n-1}$. When the point $p$
and the integral manifold $\Lambda_p$ are clear from the context, we
will refer to the deflicity of $X$ with respect to $\rho$ at $p$
simply as the deflicity of $P,R$.

\subsection{Analytic expression for the deflicity}
  
Since Noetherian functions are analytic, the germ $X$ admits a
decomposition into irreducible analytic components. Let $\{\gamma_i\}$
denote the components of $X$ which are curves and such that
$\rho\rest{\gamma_i}\not\equiv\const$ (where each curve is counted
with an associated multiplicity $m_i$). We will call these curves the
\emph{good curves} of $X$ with respect to $\rho$. All other components
(whether they are curves on which $\rho$ is constant or
higher-dimensional components) will be called \emph{bad components}.

To motivate this definition, note that since $\rho$ is an analytic
function, it cannot admit isolated zeros on components of $X$ that
have dimension greater than one, and therefore the bad components do
not contribute any isolated zeros in the definition of the deflicity
of $X$ with respect to $\rho$. On the other hand, the number of
solutions of $\rho=y$ on a good curve $\gamma_i$ converging to $p$ as
$y\to\rho(p)$ is equal by definition to $\mult_{\gamma_i}(\rho-\rho(p))$. Since each
good curve is transversal to $\rho^{-1}(\e)$ for all sufficiently
small $\e$, each solution of $\rho=\e$ on $\gamma_i$ should be counted
with multiplicity $m_i$. To conclude, we have following proposition.

\begin{Prop}\label{prop:def-of-deflicity}
  The deflicity of $X$ with respect to $\rho$ at the point $p$ is given by
  \begin{equation}
    \sum_i m_i \mult_{\gamma_i}(\rho-\rho(p))
  \end{equation}
  where the sum is taken over the good components $\gamma_i$.
\end{Prop}

\subsection{The set of non-isolated intersections}

Recall that $P=(P_1,\ldots,P_{n-1})$ and $R$ denote polynomials of
degree bounded by $d$.

\begin{Def}
  The set of non-isolated intersection of $P$ with respect to $R$,
  denoted by $\NI(P;R)$, is the set of all points $q\in\C^{n+m}$ such
  that $q$ belongs to a bad component of the set
  $\{P\rest{\Lambda_q}=0\}$ with respect to $R\rest{\Lambda_q}$.
\end{Def}

The union of the bad component on each particular integral manifold
$\Lambda_q$ is analytic, but generally not algebraic. However, the
following proposition shows that all of these sets together do form an
algebraic set.

\begin{Prop} \label{prop:ni-algebraic}
  The set $\NI(P;R)$ is algebraic. Moreover, it can be defined by
  equations of degrees not exceeding
  \begin{equation}
    d_\NI = \max[d_\IL,d_M(n,\delta,d,k)]  \qquad \text{where } k=\cN(m,n,\delta,d;0)
  \end{equation}
  where $d_\IL,d_M$ are as given
  in~\eqref{eq:gk-integral},~\eqref{eq:deg d_M} respectively.
\end{Prop}

\begin{proof}
  A point $q\in\C^{n+m}$ belongs to $\NI(P;R)$ if~\eqref{eq:noetherian-system} is
  integrable at the point, and the system of equations
  \begin{equation}
    P\rest{\Lambda_q}=(R-R(q))\rest{\Lambda_q}=0
  \end{equation}
  admits a non-isolated solution through $q$. We can write down equations
  of degree $d_\IL$ to guarantee integrability by Theorem~\ref{thm:gk-integral}.

  Since the maximal possible multiplicity of an isolated solution is
  bounded by $k$, a solution is non-isolated if and only if it has
  multiplicity greater than this number. By
  Proposition~\ref{prop:mo-fundamental} this is equivalent to the
  vanishing of $\mo(P,R-R(q))$ for every multiplicity operator $\mo$
  of order $k$. It remains to note that these expressions evaluate to
  polynomials (as functions of $q$) of the required degrees.
\end{proof}

The real analytic function $\dist(\cdot,\NI(P;R))$ will play a key role
in our arguments. We will refer to it as the \emph{critical distance}.

\subsection{The inductive step}

We will prove our main result, Theorem~\ref{thm:main}, by induction on
the dimension of the set $\NI(P;R)$. To facilitate this induction we
let $\cM(m,n,\delta,d;e)$ denote the maximum possible deformation
multiplicity at $p$ for any system~\eqref{eq:noetherian-system} with
parameters $m,n,\delta$ and any pair $X,\rho$ defined as in the beginning
of~\secref{sec:proof-structure} with $\deg P\le d$ and
$\deg R\le d$, and such that the dimension of $\NI(P;R)$ near $p$ is
bounded by $e$.

The case $e=0$ corresponds to systems where the intersection at $p$ is
in fact isolated. In this case the deformation multiplicity is given
by the usual multiplicity of $p$ as a solution of the equations
$P=R-R(p)=0$.  Thus $\cM(m,n,\delta,d;0)$ is bounded by
Theorem~\ref{thm:gk-mult}.

Theorem~\ref{thm:main} now follows by an easy induction from the
following.
\begin{Thm} \label{thm:NI-induction}
  The quantity $\cM(m,n,\delta,d;e)$ satisfies
  \begin{equation}
    \cM(m,n,\delta,d;e) \le \cM(m,n,\delta, d',e-1)
  \end{equation}
  where
  \begin{equation}
    d' = (A d_E+B)(k+1)+n+m,
  \end{equation}
  $A,B$ are defined in Proposition~\ref{prop:MRvsDist}, $d_E$ in 
  Proposition~\ref{prop:E-construction} and $k=\cN(m,n,\delta,d;0)$ as in 
  Proposition~\ref{prop:deflicity-preserving}. More explicit,
  $$
  d'\le \left(\max\{d,\delta\}\right)^{32(m+n)^4}(m+n)^{40(m+n)^5}.
  $$
\end{Thm}

We now present a schematic proof of this statement assuming the
validity of Propositions~\ref{prop:deflicity-preserving},
\ref{prop:MRvsDist}, \ref{prop:generic-def}
and~\ref{prop:E-construction}. The reader is advised to review this
proof before reading the details of the aforementioned propositions in
order to gain perspective on the context in which they are used. The
propositions are presented in the three subsections
of~\secref{sec:good-perts} and in~\secref{sec:minorize-distance}
respectively, and each can be read independently of the others.

\begin{proof}
  Consider a system~\eqref{eq:noetherian-system} with parameters
  $(n,m,\delta)$, an integral manifold $\Lambda_p$, and a pair $P,R$
  of degree $d$ such that $\dim\NI(P;R)=e$. Let $\{\gamma_i\}$ denote
  the good curves of $P,R$ on $\Lambda_p$ and let $\{\NC_i\}$ denote
  the irreducible components of $\NI(P;R)$. We assume that
  $R\rest{\Lambda_p}\not\equiv\const$ (otherwise there is nothing to
  prove). We will establish the claim by constructing a new pair
  $P',R$ of degree $d'$ such that
  \begin{enumerate}
  \item[(i).] The deflicity of $P',R$ is no smaller than of $P,R$ on $\Lambda_p$.
  \item[(ii).] We have $\dim\NI(P';R)<\dim\NI(P;R)=e$.
  \end{enumerate}
  Proposition~\ref{prop:deflicity-preserving} establishes a condition
  under which a perturbation of $P$ does not decrease the deflicity.
  Roughly, the order of the perturbations along each good curve $\gamma_i$
  must be greater than a certain prescribed order $\nu_i$. Under this
  condition, we show by a Rouch\'e principle argument that the perturbation
  cannot decrease the number roots converging to $p$.

  Proposition~\ref{prop:MRvsDist} relates the orders $\nu_i$ to the
  critical distance. Namely, we give two explicit constants $A,B$ such
  that for any good curve $\gamma_i$ the order of the critical distance
  function on $\gamma_i$ is at least $\frac{\nu_i-B}{A}$. Thus, to
  satisfy the requirements of
  Proposition~\ref{prop:deflicity-preserving}, it essentially suffices
  to find a function $E$ minorizing the critical distance.
  
  The construction of $E$ is carried out in
  Proposition~\ref{prop:E-construction}, where we construct a
  polynomial $E\in\C_{n+m}$ of an explicitly bounded degree $d_E$
  which minorizes the critical distance along each good curve and does
  not vanish identically on any of the $\NC_i$. We choose a generic
  affine-linear functional $\ell\in(\C^{n+m})^*$ which vanishes at
  $p$, and does not vanish identically on any of the $\NC_i$ and
  define
  \begin{equation}
    E' = E^A \cdot \ell^B
  \end{equation}
  where $A,B$ are the coefficients given in
  Proposition~\ref{prop:MRvsDist}. Then $E'$ satisfies the asymptotic
  conditions of Proposition~\ref{prop:deflicity-preserving}.

  Our second goal is condition (ii). We achieve this condition
  in two steps. Namely, it would clearly suffice to construct a perturbation $P'$
  satisfying the following two conditions:
  \begin{enumerate}
  \item[(A).] $\NI(P';R)\subseteq\NI(P;R)$.
  \item[(B).] None of the components $\NC_i$ are contained in
    $\NI(P';R)$.
  \end{enumerate}
  We establish condition (A) by a Sard-type argument, essentially
  using the fact that the occurrence of a non-isolated intersection is
  a condition of infinite codimension. Specifically, if we choose
  $Q_1,\ldots,Q_{n-1}$ to be sufficiently generic polynomials of
  degree $n+m$ and define
  \begin{equation}
    P'_j = P_j+Q_j(E')^{k+1}
  \end{equation}
  then according to Proposition~\ref{prop:generic-def} condition (A)
  is satisfied. Moreover, $P'$ clearly satisfies the asymptotic
  conditions of Proposition~\ref{prop:deflicity-preserving} and
  hence condition (i).

  It remains to establish condition (B). Consider a component $\NC_i$.
  Recall that $E'$ does not vanish at generic points of $\NC_i$ by
  construction. And $Q_1$, being chosen to be sufficiently generic,
  certainly does not either. Since $P_1$ vanishes at every point of
  $\NC_i$ by definition, we see that $P'_1$ does not vanish at generic
  points of $\NC_i$, and hence condition (B) is satisfied. This
  concludes the proof.
\end{proof}

\begin{proof}[Proof of Theorem~\ref{thm:main}]
We claim that in notations of Theorem~\ref{thm:NI-induction},
\begin{equation}\label{eq:bound on induction}
d'\le \left(\max\{d,\delta\}\right)^{32(m+n)^4}(m+n)^{40(m+n)^5}.
\end{equation}

Indeed, rough estimate for $\cN(m,n,\delta,d;0)$ in Theorem~\ref{thm:gk-mult} 
gives 
\begin{equation}\label{eq:bound on gk mult}
\cN(m,n,\delta,d;0)< (\delta+d)^{8(m+n)^2}(m+n)^{8(m+n)^3}.
\end{equation}

Let denote the right hand side by $C$. We have $B,(k+1)\le C$ as defined in 
Proposition~\ref{prop:MRvsDist}. 
Then $d_M(n,\delta,d,B)< C^{n+1}(d+\delta)$ by \eqref{eq:deg d_M}, and 
$d_\IL< \delta^{2(m+n)}(m+n)^{6(m+n)}$, by \eqref{eq:gk-integral}. 
Therefore, according to Proposition~\ref{prop:ni-algebraic}, $d_\NI< 
C^{m+n}$ and $A< C^{(m+n)^2}$.

From \eqref{eq:d_H} we have $d_H< C^{m+n}$. Therefore, from \eqref{eq: deg 
d_E}, $d_E< C^{2(m+n)^2}$. Therefore 
$$
d'<\left(C^{3(m+n)^2}+C\right)C+m+n<C^{4(m+n)^2}\le 
\left(\max\{d,\delta\}\right)^{32(m+n)^4}(m+n)^{40(m+n)^5}.
$$

Now, applying \eqref{eq:bound on induction} at most $n$ times, we get
$$
\cM(m,n,\delta,d;e) \le \cM(m,n,\delta, \tilde{d},0), \quad
\tilde{d}\le \left(\max\{d,\delta\}\right)^{(m+n)^{8n}}(m+n)^{(m+n)^{20n}},
$$
so, by  \eqref{eq:bound on gk mult}, we get the required bound.
\end{proof}

\section{Good perturbations}
\label{sec:good-perts}

Let $P=(P_1,\ldots,P_{n-1})$ and $R$ be polynomials of degree bounded
by $d$, and let $\Lambda_p$ be the germ of an integral manifold
of~\eqref{eq:noetherian-system} at the point $p$. Let $\{\gamma_i\}$
denote the set of good curves of $P,R$ through the point $p$ with
associated multiplicities $m_i$. Let $\ell\in\(\C^n\)^*$ be a generic linear
functional on the $x$-plane which is non-vanishing on vectors tangent
to each $\gamma_i$ at $p$, and let $\cT=\ker\ell$. At generic point
$q$ of each $\gamma_i$, the multiplicity $m_i$ is equal to the
multiplicity of the isolated intersection $P=0$ with the
$(n-1)$-dimensional $\cT$-plane. In particular, it is bounded by
$k=\cN(m,n,\delta,d;0)$.

Let $M:=\mo_\cT$ be a generic $k$-th order multiplicity operator
through the $\cT$ plane. In particular, we require that $M$ does not
vanish identically on any of the $\gamma_i$ (which is certainly
possible by Proposition~\ref{prop:mo-fundamental}) and moreover that
the order of $M$ along each $\gamma_i$ is the minimal possible. It is
easy to verify that these conditions are satisfied by $\cT,M$ outside
a proper algebraic subset of the class of $(n-1)$-planes and
multiplicity operators of order $k$ (since having order greater than a
given value is given by the vanishing of certain algebraic expressions
depending on the coefficients of $\cT,M$).

\begin{Rem}
  In fact, since the multiplicity is taken through the
  $n-1$-dimensional plane $\cT$, and the restriction of $P$ to $\cT$
  may be though of as Noetherian function in the ambient space
  $\C^{n+m-1}$ cut out by $\ell=0$, we could have chosen
  $k=\cN(m,n-1,\delta,d;0)$. However, since this does not affect our
  overall estimates and complicates the notation we use the
  weaker bound above.
\end{Rem}

\subsection{Deflicity preserving perturbations}

Let $E$ and $Q_1,\ldots,Q_{n-1}$ be holomorphic functions on
$\Lambda_p$.

\begin{Prop} \label{prop:deflicity-preserving}
  Suppose that for every $\gamma_i$
  \begin{equation} \label{eq:E-growth-condition}
    \ord_{\gamma_i}E > \max[\ord_{\gamma_i}M, \ord_{\gamma_i}(R-R(p))]
  \end{equation}
  and let
  \begin{equation}
    P_j' = P_j+Q_jE^{k+1} \quad j=1,\ldots,n-1.
  \end{equation}
  Then the deflicity of $P',R$ is no smaller than the deflicity
  of $P,R$.
\end{Prop}

\begin{proof}
  We assume for simplicity of the notation that $\ell(p)=R(p)=0$. We
  may also assume that $P,Q,E,R$ are all normalized to have norm
  bounded by $1$ in some fixed ball in $\Lambda_p$.

  Recall that $\ell$ is transversal to each $\gamma_i$ at $p$. We say
  that the \emph{pro-branches} of $\gamma_i$ with respect to $\ell$
  are the irreducible components $\gamma_{ij}$ of the set
  $\gamma_i\cap\ell^{-1}(\R_{\ge0})$. Each pro-branch is an
  irreducible germ of a real analytic curve. For $s\in(R_{\ge0},0)$
  there exists a unique point
  $\gamma_{ij}(s)\in\gamma_{ij}\cap\ell^{-1}(s)$, defining an analytic
  parametrization of $\gamma_{ij}$.

  From Proposition~\ref{prop:def-of-deflicity} and standard theory of
  holomorphic curves we see that the deflicity of $P,R$ is equal to
  \begin{equation} \label{eq:def-gamma-ij}
    \sum_{\gamma_i}\mult_{\gamma_i}R = \sum_{\gamma_{ij}} \ord_{\gamma_{ij}} R
  \end{equation}
  Here, if a curve $\gamma_i$ appears with multiplicity $m_i$ then we
  consider each of its pro-branches as appearing $m_i$ times in the
  set $\gamma_{ij}$. Similarly, if we let $\gamma'_i$ denote the good
  curves corresponding to $P',R$ and $\gamma'_{ij}$ their pro-branches
  (taking multiplicities into account), then the deflicity of $P',R$
  is equal to
  \begin{equation} \label{eq:def-gamma'-ij}
    \sum_{\gamma'_i}\mult_{\gamma'_i}R = \sum_{\gamma'_{ij}} \ord_{\gamma'_{ij}} R
  \end{equation}
  
  We will show that if $E$ satisfies the growth
  condition~\eqref{eq:E-growth-condition} then there is an injective
  map $\iota$ from the set $\{\gamma_{ij}\}$ to the set
  $\{\gamma'_{ij}\}$ such that
  \begin{equation} \label{eq:iota}
    \ord_{\gamma_{ij}}R=\ord_{\iota(\gamma_{ij})}R.
  \end{equation}
  Therefore the right-hand side of~\eqref{eq:def-gamma'-ij} is no
  smaller than the right-hand side of~\eqref{eq:def-gamma-ij}, and the
  conclusion of the proposition follows.

  Introduce an equivalence relation $\sim$ on $\{\gamma_{ij}\}$ by letting
  $\gamma_{ij}\sim\gamma_{i'j'}$ if and only if
  \begin{equation} \label{eq:cluster-def}
    \ord_s \dist(\gamma_{ij}(s),\gamma_{i'j'}(s)) > \max[\ord_{\gamma_{ij}}M,\ord_{\gamma_{ij}}R]
  \end{equation}
  In other words, two pro-branches are equivalent if and only if
  the distance between them is smaller by an order of magnitude
  than $M$ and $R$ (evaluated at one of them). It is easy to check
  that this is indeed an equivalence relation. We call the equivalence
  classes $C_\alpha$ of $\sim$ \emph{clusters}.
 
  Choose a representative $\gamma_\alpha$ for each cluster $C_\alpha$.
  By~\eqref{eq:cluster-def} we can choose an order $\nu_\alpha\in\Q$ such that
  \begin{equation} \label{eq:nu-cond1}
    \min_{\gamma_{ij}\sim\gamma_\alpha} [ \ord_s 
    \dist(\gamma_\alpha(s),\gamma_{ij}(s)) ]
    > \nu_\alpha > 
    \max[\ord_{\gamma_\alpha}M,\ord_{\gamma_\alpha}R],
  \end{equation}
  and by~\eqref{eq:E-growth-condition} we can also require that
  \begin{equation} \label{eq:nu-cond2}
    \ord_{\gamma_\alpha} E > \nu_\alpha.
  \end{equation}

  Consider the ball $B_\alpha(s)\subset\{\ell=s\}$ with center at $\gamma_\alpha(s)$
  and radius $s^{\nu_\alpha}$. By~\eqref{eq:nu-cond1} and the definition of $\sim$
  it follows that for $s$ sufficiently small, $B_\alpha(s)$ meets the
  pro-branches $\gamma_{ij}$ in the cluster $C_\alpha$ and only them.

  Given~\eqref{eq:nu-cond1} and~\eqref{eq:nu-cond2},
  Corollary~\ref{cor:mo-power-pert} and the
  subsequent~\eqref{eq:mo-power-rouche} apply with
  $r=s^{\nu_\alpha}/A''_{n,k}$ for $s$ sufficiently small. It follows
  that the number of zeros of $P'=0$ in the ball
  $B'_\alpha(s)\subset\{\ell=s\}$ with center at $\gamma_\alpha(s)$
  and radius $s^{\nu_\alpha}/A''_{n,k}$ is at least the size of the
  cluster $C_\alpha$.
  
  By~\eqref{eq:nu-cond1} the balls $B'_\alpha(s)$ are disjoint for
  different $\alpha$ and sufficiently small $s$. Consider now the
  pro-branches $\gamma'_{ij}$. We will say that $\gamma'_{ij}$
  \emph{lies in $B'_\alpha$} if $\gamma'_{ij}(s)\in B'_\alpha(s)$ for
  sufficiently small $s$. By the above, at least
  $\#C_\alpha$ of the pro-branches must lie in $B'_\alpha$. Let $\iota$ be an
  arbitrary injection from the pro-branches $\gamma_{ij}$ belonging to
  $C_\alpha$ to the pro-branches $\gamma'_{ij}$ lying in $B'_\alpha$.

  The construction will be finished if we prove that $\iota$
  satisfies~\eqref{eq:iota}. But this is clear, since
  by~\eqref{eq:nu-cond1} the radius of $B'_\alpha(s)$ is an order of
  magnitude smaller than $R$ on $\gamma_\alpha$ and therefore any
  pro-branch $\gamma'_{ij}$ lying in $B'_\alpha$ must satisfy
  \begin{equation}
    \ord_{\gamma'_{ij}}R = \ord_{\gamma_\alpha} R.
  \end{equation}
\end{proof}

\subsection{Minorizing the growth conditions by the critical distance}

Proposition~\ref{prop:deflicity-preserving} allows us to construct
a perturbation of $P,R$ which does not decrease the deflicity, given
a function $E$ which is small on the good curves
$\gamma_i$ compared to $R$ and $M$.

In this subsection we show that $R$ and $M$ are minorized by the
critical distance, and hence in order to apply
Proposition~\ref{prop:deflicity-preserving} it suffices to minorize
the critical distance. More precisely,

\begin{Prop} \label{prop:MRvsDist}
  For every good curve $\gamma_i$ we have
  \begin{equation}
    \max[\ord_{\gamma_i}M, \ord_{\gamma_i}(R-R(p))] \le  A \ord_{\gamma_i} \dist(\cdot,\NI(P,R-R(p))) + B
  \end{equation}
  where
  \begin{eqnarray*}
    A &:=& \max[d_\NI,d_M(n,\delta,d,B)]^{n+m} \\
    B &:=& \cN(m,n,\delta,d;0)
  \end{eqnarray*}
\end{Prop}

\begin{proof}
  We start with the estimate for the function $R-R(p)$. Let
  $k=\cN(m,n,\delta,d;0)$ and let $I_R$ generated by the polynomials
  $\mo(P,R-R(p))$ (where $\mo$ ranges over all multiplicity operators
  of order $k$) and the polynomials defining the
  integrability locus $\IL$ provided by Theorem~\ref{thm:gk-integral}.
  We denote the set of all of these generators by $\{G_i\}$.

  Arguing in the same manner as in the proof of
  Proposition~\ref{prop:ni-algebraic}, we see that the zero locus
  of $I_R$ is contained in $\NI(P;R)$ (in fact, it consists of those
  points in $\NI(P;R)$ where $R=R(p)$).

  By the effective \Lojas. inequality \cite{Kollar:Lojas}, there exists a
  constant $C>0$ such that
  \begin{equation}
    \max_i \abs{G_i(\cdot)} \ge C \dist(\cdot,Z(I_R))^A \ge
      C \dist(\cdot,\NI(P;R))^A.
  \end{equation}
  Let $\gamma_i$ be a good curve. Then $\gamma_i\subset\IL$ and the
  polynomials defining $\IL$ vanish identically on it. Therefore along
  $\gamma_i$ the maximum must be attained for one of the generators
  given by the multiplicity operator $\mo$, and hence
  \begin{equation}
    \ord_{\gamma_i} \mo(P,R-R(p)) \le  A \ord_{\gamma_i} \dist(\cdot,\NI(P;R)).
  \end{equation}
  Since $P_1,\ldots,P_{n-1}$ vanish identically on $\gamma_i$, the
  conclusion now follows by application of
  Theorem~\ref{thm:mo-curve-growth}.

  Proceeding now to the estimate for $M$, recall that we have
  $M=\mo_\cT(P)$ for a generic multiplicity operator $\mo_\cT$ where
  the direction $\cT$ was chosen to be transversal to the good curves
  on $\Lambda_p$. However, in principle $\cT$ may be parallel to good
  curves on other integral manifolds (i.e., its translate may contain
  them).

  To avoid this problem we let $\cT_1,\ldots,\cT_n$ be a tuple of
  linearly independent $(n-1)$-dimensional subspaces of the $x$-plane,
  each satisfying the same genericity condition as $\cT$. For any
  integral manifold $\Lambda$ and any point $q\in\{P\rest\Lambda=0\}$,
  if $q\not\in\NI(P;R)$ then the zero locus of $P\rest\Lambda$ near
  $q$ must be a curve, and at least one of $\cT_1,\ldots,\cT_n$ must
  not be parallel to this curve. Then there exists a multiplicity
  operator of order $k$ in the $\cT_i$ direction which is not
  vanishing at $q$.

  Let $I_M$ denote the ideal generated by the polynomials defining the
  integrability locus and all multiplicity operators of order $k$
  through any of spaces $\cT_1,\ldots,\cT_n$ of $P$. By the above, the
  zero locus of $I_M$ is contained in $\NI(P;R)$. Moreover, the
  integrability conditions vanish identically on $\gamma_i$, and all
  other generators of $I_M$ have orders no smaller than the order of
  $\mo_\cT(P)$ (since our original choice of $\cT$ and $\mo_\cT$ was
  generic). We can now complete the proof in a manner analogous to the
  argument we used for $I_R$. We leave the details for the reader.
\end{proof}

\begin{Rem}\label{rem:why-not-gk}
  Proposition~\ref{prop:MRvsDist} is the only step where we
  essentially need the assumption that the deformations under
  consideration are Noetherian with respect to the deformation
  parameter. Namely, in order to apply the effective \Lojas.
  inequality we require that the entire deformation space be
  algebraic, rather than the weaker condition that each fiber
  $R^{-1}(\e)$ be separately algebraic as provided in the
  Gabrielov-Khovanskii conjecture.
\end{Rem}

\subsection{Sard-type claim for generic perturbations}

We are interested in applying
Proposition~\ref{prop:deflicity-preserving} in order to produce
a perturbation of $P,R$ which does not decrease deflicity,
and which reduces the set of non-isolated intersections $\NI(P;R)$.
The first step is to show that our perturbation does not create
new non-isolated intersections. We will show that for a sufficiently
generic choice of the coefficients $Q_j$ this will be the case.

Let $E$ be a Noetherian function.  Let $\beta\in\N$ and denote by $\cP_\beta$
the space of polynomials of degree bounded by $\beta$ in $n$
variables. Finally let $(Q_1,\ldots,Q_n)\in\cP_\beta^n$, and set
$P'_i=P_i+Q_iE$ for $i=1,\ldots,n$.

\begin{Prop} \label{prop:generic-def} Assume
  $R\rest{\Lambda_p}\not\equiv\const$. There exists a neighborhood
  $U_p\subset\C^{n+m}$ of $p$ such that for a generic tuple
  $Q=(Q_1,\ldots,Q_n)\in\cP_{n+m}^n$ (outside a proper algebraic set),
  \begin{equation} \label{eq:generic-pert}
    \NI(P';R)\cap U_p \subseteq \NI(P;R)\cap U_p
  \end{equation}
  where $P'_i$ is defined as above.
\end{Prop}

\begin{proof}
  Fix a Euclidean ball $U_p$ centered around $p$ such that $R$ is not
  constant on any integral manifold of \eqref{eq:noetherian-system} in
  $U_p$. Then for every point in $q\in U_p\cap\IL$ either
  Lemma~\ref{lem:sard-e-nz} or Lemma~\ref{lem:sard-e-z} below is
  applicable, depending on whether $E(q)=0$, with parameters
  $\beta=n+m$ and $l=n$. In both cases it follows that there exists a
  set $B(q)\subset\cP_{n+m}^n$ of codimension $n+m+1$, such that if
  $Q\not\in B(q)$ then $q\not\in\NI(P;R)$ implies $q\not\in\NI(P';R)$.
  Moreover, as noted in the proof of Lemmas~\ref{lem:sard-e-nz}
  and~\ref{lem:sard-e-z}, the graph of the relation $Q\in B(q)$ in
  each lemma is algebraic in $q$, so that the set
  \begin{equation}
    G = \{(Q,q)\in\cP_{n+m}^n\times U_p : Q\in B(q) \},
  \end{equation}
  is real algebraic constructible. Since $\dim U_p=n+m$ it follows
  that the projection of $G$ to $\cP_{n+m}^n$ has codimension at least
  $1$. Any $Q$ outside this projection will
  satisfy~\eqref{eq:generic-pert}.
\end{proof}

It remains to state and prove the following two lemmas.

\begin{Lem} \label{lem:sard-e-nz}
  Suppose that $q\in\IL, E(q)\neq0$ and suppose further that
  $R\rest{\Lambda_q}\not\equiv R(q)$. Let $0\le l\le n$.

  For $(Q_1,\ldots,Q_l)$ outside a set $B_l$ of codimension $\beta+1$
  in $\cP_\beta^l$, the set
  \begin{equation}
    Z_l := \{ P'_1=\cdots=P'_l=0, R=R(q)\}\cap\Lambda_q
  \end{equation}
  is empty or has pure codimension $l+1$ (in a neighborhood of $q$).
\end{Lem}
\begin{proof}
  Let $B_l$ be the set of $Q_i$ violating the condition. It is defined
  by the conditions $\mo_q(F^l_1,\ldots,F^l_n)=0$ for all multiplicity
  operators $\mo$ of every order $k$, where
  $\left(F^l_i\right)_{i=1}^n$ are $n$-tuples consisting of
  $P'_1\rest{\Lambda_q},\cdots,P'_l\rest{\Lambda_q},
  R\rest{\Lambda_q}-R(q)$ and $n-l-1$ generic linear functions
  vanishing at $q$. These expressions are polynomial in the
  coefficients of $Q_1,\ldots,Q_l$, so the set $B_l$ is algebraic and
  its dimension is well defined. We note that the conditions are
  algebraic with respect to $q$ as well.

  We prove the claim by induction. The case $l=0$ corresponds
  precisely to our assumption  $R\rest{\Lambda_q}\not\equiv R(q)$.
  
  Suppose that the claim is proved for $l-1$. Consider the projection
  $\pi:\cP_\beta^l\to\cP_\beta^{l-1}$ forgetting the last coordinate.
  The set $\pi^{-1}(B_{l-1})$ has codimension at least $\beta+1$ by
  induction. The claim will follow if we show that the fiber of each
  point outside $B_{l-1}$ intersects $B_l$ in codimension at least $\beta+1$.

  Let $(Q_1,\ldots,Q_{l-1})\not\in B_{l-1}$. Then $Z_{l-1}$ is either
  empty or has pure codimension $l$. If it is empty, there is nothing
  to prove. Otherwise $(Q_1,\ldots,Q_l)$ will belong to the fiber if
  and only if $P'_l$ vanishes identically on some irreducible
  component of $Z_{l-1}$. Since there are finitely many such
  components, it will suffice to check that identical vanishing on
  each of them has codimension at least $\beta+1$. Let $Z'_{l-1}$ be one
  such component.

  Since $E(q)\neq0$, we have $P'_l\rest{Z'_{l-1}}\equiv0$ if and only if
  $Q_l\equiv - P_l/E$ identically on $Z'_{l-1}$. This is an affine-linear
  condition on $Q_l$ of codimension at least $\dim \cP_\beta\rest{Z'_{l-1}}$. It remains
  only to note that this dimension is at least $\beta+1$: for instance
  if $x_j$ is a coordinate not identically vanishing on $Z'_{l-1}$ then
  clearly $1,x_j,\ldots,x_j^\beta$ are linearly independent as functions
  defined on $Z'_{l-1}$.
\end{proof}

\begin{Lem} \label{lem:sard-e-z}
  Suppose that $q\in\IL,q\not\in\NI(P;R)$ and $E(q)=0$. Let $0\le l\le n$.

  For $(Q_1,\ldots,Q_l)$ outside a set $B_l$ of codimension $\beta+1$
  in $\cP_\beta^l$, the set
  \begin{equation}
    Z_l := \{ P'_1=\cdots=P'_l=P_{l+1}=\cdots=P_n=0, R=R(q)\}\cap\Lambda_q
  \end{equation}
  is zero-dimensional in a neighborhood of $q$ (that is, contains
  only, possibly, $q$).
\end{Lem}

\begin{proof}
  Let $B_l$ be the set of $Q_i$ violating the condition. We can check
  that $B_l$ is algebraic as in the proof of
  Lemma~\ref{lem:sard-e-nz}.

  We prove the claim by induction. The case $l=0$ corresponds
  precisely to our assumption that $q\not\in\NI(P;R)$.
  
  Suppose that the claim is proved for $l-1$. Consider the projection
  $\pi:\cP_\beta^l\to\cP_\beta^{l-1}$ forgetting the last coordinate.
  The set $\pi^{-1}(B_{l-1})$ has codimension at least $\beta+1$ by
  induction. The claim will follow if we show that the fiber of each
  point outside $B_{l-1}$ intersects $B_l$ in codimension at least
  $\beta+1$.

  Let $(Q_1,\ldots,Q_{l-1})\not\in B_{l-1}$. Then the set $Z_{l-1}$ is
  zero dimensional (in a neighborhood of $q$). If it is in fact empty,
  then one of the equations defining it is non-vanishing at $q$. In
  this case, since $E(q)=0$ by assumption, $Z_l$ is empty as well.

  We therefore must consider the case that $Z_{l-1}=\{q\}$ (in a
  neighborhood of $q$). In this case, the equations
  \begin{equation}
    \{P'_1=\cdots=P'_{l-1}=P_{l+1}=\cdots=P_n=0, R=R(q)\}\cap\Lambda_q
  \end{equation}
  define a curve $\gamma\subset\Lambda_q$. Thus $(Q_1,\ldots,Q_l)$
  will belong to $B_l$ if and only if $P'_l$ vanishes identically on
  an irreducible component of this curve. Since there are finitely
  many such components, it will suffice to check that identical
  vanishing on each of them has codimension at least $\beta+1$. Let
  $\gamma'$ be one such component.

  We have $P'_l\rest{\gamma'}\equiv0$ if and only if $EQ_l\equiv - P_l$
  identically on $\gamma'$. This is an affine-linear condition on
  $Q_l$ of codimension at least $\dim \cP_\beta\rest{\gamma'}$: if
  $E\not\equiv0$ on $\gamma'$ then this is clear, and otherwise the
  condition is never satisfied because $S_{l-1}$ is a regular sequence
  and hence $P_l$ does not vanish identically on $\gamma$. The
  conclusion now follows as in the proof of Lemma~\ref{lem:sard-e-nz}.
\end{proof}

\section{Minorizing the critical distance}
\label{sec:minorize-distance}

Let $P=(P_1,\ldots,P_{n-1})$ and $R$ be polynomials of degree bounded
by $d$, and let $\Lambda_p$ be the germ of an integral manifold
of~\eqref{eq:noetherian-system} at the point $p$. Let $\{\gamma_i\}$
denote the set of good curves of $P,R$ through the point $p$ with
associated multiplicities $m_i$.

Our goal in this section is to construct a Noetherian function $E$ of
bounded degree such that $E$ minorizes the critical distance on good curves, 
and does not vanish identically on any of the top-dimensional components of
$\NI(P;R)$. The main step in the construction is the following Lemma.

We introduce some notations to facilitate our proof. If $\cT$ is a
linear subspace of the $x$-coordinates, then for every point $q\in\IL$
we will denote by $\cT_q\subset\Lambda_q$ the integral submanifold
through $q$ of the sub-distribution of~\eqref{eq:noetherian-system}
corresponding to $\cT$. Similarly $B_\cT(q,r)\subset\cT_q$ will denote
the ball of radius $r$ around $q$ in $\cT_q$.

The $x$-plane provides natural coordinates on the integral manifolds
of~\eqref{eq:noetherian-system} in a neighborhood $U_p$ of $p$. In particular
for any two integral manifolds $\Lambda_{q_1},\Lambda_{q_2}\subset U_p$ we
have a map $\tau^{q_1}_{q_2}:\Lambda_{q_1}\to\Lambda_{q_2}$ mapping each point in
$\Lambda_{q_1}$ to the point with the same $x$-coordinates in $\Lambda_{q_2}$.
If we choose $U_p$ small enough, then by the analytic dependence of flows
on initial conditions, for every $q\in U_p$ we have
\begin{equation} \label{eq:tau-dist}
  \dist(q,\tau^{q_1}_{q_2}q) \le 2 \dist(q_1,q_2).
\end{equation}

\begin{Lem} \label{lem:H-construction}
  Let $\NC$ be an irreducible component of $\NI(P;R)$. There exists
  a polynomial $H\in\C_{n+m}$ such that
  \begin{enumerate}
  \item For every good curve $\gamma$ we have
    \begin{equation}
      \ord_\gamma H \ge \ord_\gamma \dist(\cdot,\NC)
    \end{equation}
  \item $H$ does not vanish identically on $\NC$.
  \item The degree of $H$ is bounded by
    \begin{equation}\label{eq:d_H}
      d_H := (k+1)d_M(n,\delta,d,k) \qquad \text{where } k=\cN(m,n,\delta,d;0)
    \end{equation}
  \end{enumerate}
\end{Lem}

\begin{proof}
  For any $q\in\NC$ the set $\NI(P;R)\cap\Lambda_q$ consists of those
  components of $\{P=0\}\cap\Lambda_q$ which have dimension greater
  than $1$, or where $R$ is constant. In particular,
  $\NC\cap\Lambda_q$ is a union of such components. Assume now that
  $q'\in\NC$ is generic. Then locally near $q'$ we will have
  $\NC\cap\Lambda_{q'}=\{P=0\}\cap\Lambda_{q'}$. Thus if
  $\codim_{\Lambda_{q'}}(\NC\cap\Lambda_{q'})=l$, then letting
  $\Phi=(P_1,\ldots,P_l)$ (up to a reordering of $P_j$), we have
  $\NC\cap\Lambda_{q'} = \{\Phi=0\}\cap\Lambda_{q'}$
 
  Let now $\cT$ be a generic $l$-dimensional subspace of the $x$
  coordinates. Denote by $\lambda$ the multiplicity of the isolated
  zero $\{\Phi\rest{\cT_{q'}}=0\}$. In particular, $\lambda \le k$. By
  Proposition~\ref{prop:mo-fundamental} there exists a multiplicity
  operator $\mo[\lambda]$ of order $\lambda$ such that $\mo[\lambda]_{\cT,q'}(\Phi)\neq0$. Let
  $M=\mo[\lambda]_\cT(\Phi)$.

  \begin{Claim} \label{claim:M-NC} At any point $q\in\NC$ we have
    $\mult_q \Phi\rest{\cT_q} \ge \lambda$. Moreover, if $M(q)\neq0$ then
    \begin{enumerate}
      \item $\mult_q \Phi\rest{\cT_q} = \lambda$.
      \item $\NC\cap\Lambda_q = \{\Phi=0\}\cap\Lambda_q$ in a neighborhood of $q$.
    \end{enumerate}
  \end{Claim}
  \begin{proof}[Proof of the claim]
    The first claim $\mult_q \Phi\rest{\cT_q} \ge \lambda$ holds
    because $\lambda$ was chosen as the generic (and hence minimal)
    multiplicity of a zero of $\Phi\rest{\cT_{q'}}$ at a point of
    $\NC$. Assume now that $M(q)\neq0$. Then
    $\mult_q \Phi\rest{\cT_q}$ cannot exceed $\lambda$ by
    Proposition~\ref{prop:mo-fundamental}. In particular it is finite,
    so the set $\{\Phi=0\}\cap\Lambda_q$ has codimension $l$ near $q$.
    Since $l$ was chosen to be the generic (hence maximal) codimension
    of $\NC$ intersected with any integral manifold, the codimension
    of $\NC\cap\Lambda_q$ is at most $l$. Since
    $\NC\cap\Lambda_q\subset\{\Phi=0\}\cap\Lambda_q$ we see that
    $\NC\cap\Lambda_q$ must in fact be a union of irreducible
    components of $\{\Phi=0\}\cap\Lambda_q$. Suppose toward
    contradiction that this set has another component $\cC$ through
    $q$.

    Since $M(q)\neq0$, the set $\{\Phi=0\}\cap\Lambda_q$ has an
    isolated intersection with $\cT_q$ at $q$. Consider generic
    $q'\in\Lambda_q$ arbitrarily close to $q$. Then $\cT_{q'}$ will
    meet $\cC$ at some point $q_1$ close to $q$, and the set
    $\NC\cap\Lambda_q$ at some other point $q_2$ close to $q$. Both
    points $q_1,q_2$ correspond to zeros of $\Phi\rest{\cT_{q'}}$ and
    by the first part of the claim
    $\mult_{q_2} \Phi\rest{\cT_{q_2}}\ge\lambda$. As $q'\to q$ both
    $q_1,q_2\to q$ and hence $\mult_q \Phi\rest{\cT_q}>\lambda$
    contradicting what was already proved.
  \end{proof}

  Let $\gamma$ be any good curve, and denote by $\gamma(t)$ a
  pro-branch with $\gamma(0)=p$. Since $\{\Phi=0\}\cap\Lambda_p$
  contains $\gamma$ it does not equal $\NC\cap\Lambda_p$ around $p$,
  and Claim~\ref{claim:M-NC} implies that $M(p)=0$. Since
  $M(\gamma(t))$ is analytic we can fix $t_0>0$ such that its modulus
  is monotone for $t\le t_0$. We claim that
  \begin{equation} \label{eq:MvCD}
    \ord_\gamma \dist(\cdot,\NC) \le (\lambda+1) \ord_\gamma M .
  \end{equation}
  Suppose to the contrary that for some $\e>0$ and for arbitrarily
  small $t$ there exist points $y(t)\in\NC$ such that
  \begin{equation} \label{eq:rho-vs-Mg}
    \rho(t) < |M(\gamma(t))|^{\lambda+1+\e}, \qquad \rho(t):=\dist(\gamma(t),y(t)).
  \end{equation}
  Then we will produce a sequence of points in $\NC$ converging to
  $\gamma(t_0)$, which is impossible since $\NC$ is closed and
  $\gamma$ is a good curve. More specifically, we analytically
  continue the point $y(t)$ to a curve in $\NC\cap\Lambda_{y(t)}$
  using (two applications of) Lemma~\ref{lem:continutation} and show
  that the endpoint of this curve converges to $\gamma(t_0)$ as
  $t\to0$.

  \begin{figure}
    \includegraphics[width=1\textwidth]{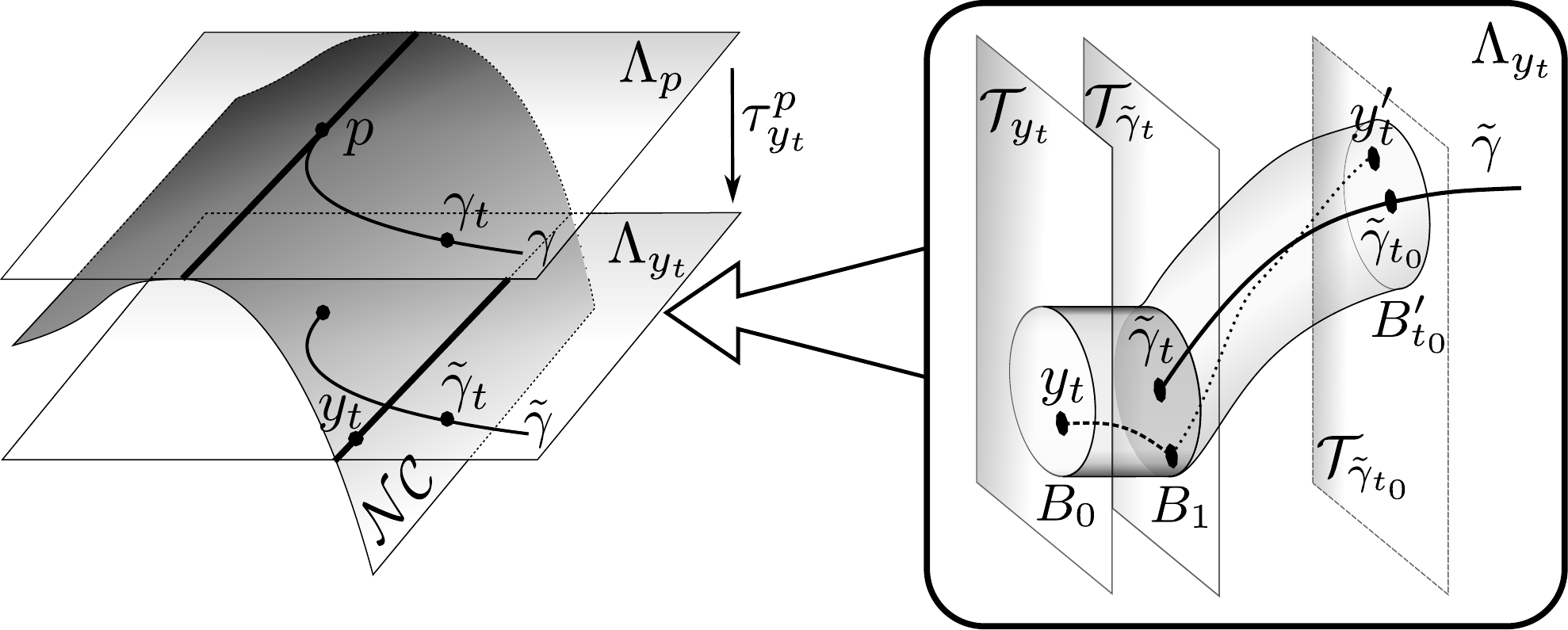}
    \caption[short]{ On the left, the entire ambient space; On the
      right, the leaf $\Lambda_{y(t)}$. The two dotted lines
      correspond to the two steps of analytic continuation. In the
      interest of space we render $\tau^p_{y(t)}(X)$ as $\tilde X$ and
      $X(t)$ as $X_t$.}
  \end{figure}

  In what follows whenever we use asymptotic class notation
  ($O,o,\Omega,\Theta$) with complex-valued functions we implicitly
  interpret it as applying to their modulus. We stress that the
  asymptotic constants are understood to be independent of $t,s$. Since
  $M$ is holomorphic in $U_p$ it is Lipschitz there, and it follows
  that for any two points $p_1,p_2\in U_p$ we have
  \begin{equation}\label{eq:m-lip}
    M(p_1)-M(p_2) = O(\dist(p_1,p_2))
  \end{equation}
  and similarly from $\Phi$. For instance, from~\eqref{eq:rho-vs-Mg}
  it follows that
  \begin{equation}\label{eq:rho-vs-My}
    \rho(t) = O(M(y(t))^{\lambda+1+\e}).
  \end{equation}
  We can thus choose positive $r(t)\in\R_+$ satisfying
  \begin{eqnarray}
    r(t) &=& o(M(y(t))), \label{eq:r-v-My} \\
    \rho(t) &=& o(r^\lambda(t)M(y(t))) \label{eq:r-v-rho-my}.
  \end{eqnarray}

  Consider the family of balls $B_s=B_\cT(c_s,r(t))\subset\Lambda_{y(t)}$
  connecting $B_0$ and $B_1$,
  \begin{equation}
    B_0=B_\cT(y(t),r(t)), \qquad B_1=B_\cT(\tau^p_{y(t)}\gamma(t),r(t))
  \end{equation}
  by a linear motion of their centers $c_s$ in the $x$-coordinates.
  Denote $\Phi_s:=\Phi\rest B_s$. Note that by~\eqref{eq:tau-dist} and
  the triangle inequality we have
  \begin{equation}\label{eq:dist-gamma-t-c}
    \dist(\gamma(t),c_s)=O(\rho(t)).
  \end{equation}
  We claim that Lemma~\ref{lem:continutation} applies to $\Phi_s$.
  Indeed,
  \begin{enumerate}
  \item[(i)] $\Phi_s$ admits at most $\lambda$ zeros for every $s\in[0,1]$.
    Indeed, by~\eqref{eq:dist-gamma-t-c} we see that
    $\dist(c_s,y(t))=O(\rho(t))$. Then~\eqref{eq:m-lip}
    and~\eqref{eq:rho-vs-My} give $M(c_s)=\Theta(M(y(t)))$. On the
    other hand, if $\Phi_s$ has more than $\lambda$ zeros in $B_s$ then by
    Theorem~\ref{thm:mo-polydisc-zeros} we have $M(c_s)=O(r(t))$. This
    contradicts~\eqref{eq:r-v-My}.
  \item[(ii)] $\Phi_s$ admits at least one zero for every $s\in[0,1]$.
    Indeed, by Theorem~\ref{thm:mo-sphere-growth} there exists a
    radius $A_{n,\lambda}r(t)<\tilde r(t)<r(t)$ such that the minimum of
    $\norm{\Phi_0}$ over $\partial B_\cT(y(t),\tilde r(t))$ is
    $\Omega(M(y(t)) \tilde r^\lambda(t))$, which by~\eqref{eq:r-v-rho-my} is
    asymptotically larger than $\rho(t)$. On the other hand,
    by~\eqref{eq:tau-dist} we may view $\Phi_s$ as a perturbation of
    size $O(\rho(t))$ of $\Phi_0$. By the Rouch\'e principle, such a
    perturbation does not change the number of zeros of $\Phi_0$ in
    $B_\cT(y(t),\tilde r(t))$. Since $y(t)$ is such a zero, $\Phi_s$
    must have a zero as well.
  \item[(iii)] The zero $y(t)$ of $\Phi_0$ has multiplicity $\lambda$ and does not
    bifurcate for small $s$. Indeed, $y(t)\in\NC$ and $M(y(t))\neq0$
    by construction, so by Claim~\ref{claim:M-NC} it must be a root of
    multiplicity $\lambda$ of $\Phi_0$. Moreover,
    $\NC\cap\Lambda_{y(t)} = \{\Phi=0\}\cap\Lambda_{y(t)}$ locally
    near $y(t)$. It follows that $y(t)$ cannot bifurcate for small
    values of $s$: it must remain an element of
    $\NC\cap\Lambda_{y(t)}$, and $\lambda$ is the minimal possible
    multiplicity for such a root.
  \end{enumerate}
  By Lemma~\ref{lem:continutation} we conclude that $\Phi_s$ has a
  zero $y_s(t)$ of multiplicity $\lambda$. Moreover by (iii) above
  $y_s(t)\in\NC\cap\Lambda_{y(t)}$ for small $s$. By analyticity the
  same must hold for every $s\in[0,1]$. Thus $y_1(t)$ is a
  zero of $\Phi_1$ and $y_1(t)\in\NC\cap\Lambda_{y(t)}$.

  Now consider the family of balls
  $B'_s=B_\cT(c'_s,r(t))\subset\Lambda_{y(t)}$ for $s\in[t,t_0]$ where
  \begin{equation}
    c'_s = \tau^p_{y(t)}\gamma(s).
  \end{equation}
  Note that $B'_t=B_1$. As before we set $\Phi'_s:=\Phi\rest {B'_s}$ and
  claim that Lemma~\ref{lem:continutation} applies to $\Phi'_s$.
  Indeed,
  \begin{enumerate}
  \item[(i')] $\Phi'_s$ admits at most $\lambda$ zeros for every $s\in[0,1]$.
    Indeed, by~\eqref{eq:tau-dist} we see that
    $\dist(c'_s,\gamma(s))=O(\rho(t))$. By~\eqref{eq:m-lip} we have
    \begin{equation} \label{eq:Mcps-v-My}
      M(c'_s) = M(\gamma(s))+O(\rho(t)) = \Omega(M(\gamma(t)))+O(\rho(t))=\Omega(M(y(t)))
    \end{equation}
    where we used the monotonicity of $M\rest\gamma$ in the second
    equality. On the other hand, if $\Phi_s$ has more than $\lambda$ zeros
    in $B'_s$ then by Theorem~\ref{thm:mo-polydisc-zeros} we have
    $M(c'_s)=O(r(t))$. This contradicts~\eqref{eq:r-v-My}.
  \item[(ii')] $\Phi'_s$ admits at least one zero for every $s\in[0,1]$.
    Indeed, let $\tilde B_s:=B_\cT(\gamma(s),r(t))$ and
    $\tilde\Phi_s=\Phi\rest{\tilde B_s}$. By~\eqref{eq:tau-dist} we
    may view $\tilde\Phi_s$ as a perturbation of size $O(\rho(t))$ of
    $\Phi_s$. Arguing as in the previous item (ii) and
    using~\eqref{eq:Mcps-v-My}, we see that this perturbation does not
    change the number of zeros in an appropriately chosen ball around
    the center. Since the center $\gamma(s)$ is a zero of
    $\tilde\Phi_s$, it follows that $\Phi_s$ must have a zero as well.
  \item[(iii')] The zero $y_1(t)$ of $\Phi'_t$ has multiplicity $\lambda$ and does
    not bifurcate for small $s$. Indeed, $\dist(y(t),y_1(t))=O(r(t))$
    and it follows using~\eqref{eq:r-v-My} that
    $M(y_1(t))=\Theta(M(y(t)))$. Since the latter is non-zero by
    construction we conclude that $M(y_1(t))\neq0$. Moreover
    $y_1(t)\in\NC$ by construction. The proof is now concluded in the
    same way as the previous item (iii).
  \end{enumerate}

  We thus apply Lemma~\ref{lem:continutation} and conclude in the same
  way as before that $\Phi'_{t_0}$ has a zero $y'(t)$, and moreover
  that $y'(t)\in\NC\cap B'_{t_0}$. As $t$ tends to zero the center of
  $B'_{t_0}$ tends to $\gamma(t_0)$ and its radius tends to zero,
  hence $y'(t)$ tends to $\gamma(t_0)$. As $y'(t)\in\NC$ we obtain the
  desired contradiction. Therefore~\eqref{eq:MvCD} is proved, and
  taking $H=M^{\lambda+1}$ concludes the proof.
\end{proof}

The proof of Lemma~\ref{lem:H-construction} will be completed once we prove
the following lemma.

\begin{Lem} \label{lem:continutation} Let $U\subset\C^l$ be an open
  domain, and $\Phi_s:U\to\C^l$ be an analytic family of holomorphic
  mappings $s\in[0,1]$. Let $\lambda\in\N$ and assume that
  \begin{enumerate}
  \item[(i)] $\Phi_s$ has at most at most $\lambda$ zeros in $U$, counting
    multiplicities, for all $s$.
  \item[(ii)] $\Phi_s$ has at least one zero in $U$ for all $s$.
  \item[(iii)] $\Phi_0$ has a zero $y_0$ of multiplicity $\lambda$ in $U$, while
    lies on a germ of a curve $y_s$ of zeros of multiplicity $\lambda$ of
    $\Phi_s$ (i.e. $y_0$ doesn't bifurcate into several zeros for
    small values of $s$).
  \end{enumerate}
  Then $y_s$ can be analytically extended to a curve of zeros of
  multiplicity $\lambda$ of $\Phi$ lying in $U$ for all $s\in[0,1]$.
\end{Lem}
\begin{proof}[Proof of the Lemma]
  Indeed, the curve $y_s$ can be analytically extended in $s$ as long
  as it doesn't leave $U$, as the non-bifurcating condition is
  analytic. Suppose toward contradiction that $y_s$ leaves $U$, and
  let $S$ denote the infimum of the set $\{s:y_s\not\in U\}$. Then by
  (ii) the map $\Phi_S$ must have some other zero $y'\in U$, and since
  $U$ is open $y'$ may be continued to a zero $y'_s$ of $\Phi_s$ for
  $s$ close to $S$. But then for $s$ slightly smaller than $S$ we have
  a zero $y_s$ of multiplicity $\lambda$ (by (iii)) and a zero $y'_s$,
  contradicting (i).
\end{proof}

We require one more standard fact, whose proof we include for the
convenience of the reader.

\begin{Fact} \label{fact:variety-eqs}
  Let $V\subset\C^N$ be an affine variety of degree $D$.
  Then $V$ is set-theoretically cut out by polynomials of degree $D$.
\end{Fact}
\begin{proof}
  Let $x\in\C^N\setminus V$. We will find a polynomial of degree
  bounded by $D$ that does not vanish at $x$. If $V$ is a hypersurface
  then it is the zero locus of a polynomial of degree $D$ and the
  claim is obvious.  Otherwise choose a generic projection
  $\pi:\C^N\to\C^{\dim V+1}$, such that $\pi(x)\not\in\pi(V)$. Then
  $\pi(V)$ is a hypersurface of degree $D$, and the previous argument
  produces a polynomial of degree $D$ which vanishes on
  $\pi^{-1}\pi(V)$. Since $x$ is not contained in this set, the proof
  is completed.
\end{proof}

Finally we can present the construction of the function $E$.

\begin{Prop} \label{prop:E-construction}
  There exists a polynomial $E\in\C_{n+m}$ such that
  \begin{enumerate}
  \item For every good curve $\gamma$ we have
    \begin{equation}
      \ord_\gamma E \ge \ord_\gamma \dist(\cdot,\NI(P;R))
    \end{equation}
  \item $H$ does not vanish identically on any irreducible component
    of $\NI(P;R)$.
  \item The degree of $E$ is bounded by
    \begin{equation}\label{eq: deg d_E}
      d_E := d_\NI^{2(n+m)}+d_H
    \end{equation}
  \end{enumerate}
\end{Prop}

\begin{proof}
  Let $\NI(P;R)=\cup_{i=1,\ldots,s}\NC_i$ be the irreducible
  decomposition of $\NI(P;R)$.  By Proposition~\ref{prop:ni-algebraic}
  the set $\NI(P;R)$ can be defined by polynomial equations of degree
  $d_\NI$. Therefore, $s\le d_\NI^{n+m}$ and any irreducible
  component $\NC_i$ of this set has degree bounded by
  $d_\NI^{n+m}$. Choose a polynomial $Q_i$ of this degree which
  vanishes on $\NC_i$ and not on any $\NC_j$ for $j\neq i$. Also
  construct for each $\NC_i$ the polynomial $H_i$ provided by
  Lemma~\ref{lem:H-construction}.

  Let
  \begin{equation} \label{eq:e-def}
    E = \sum_{i=1}^s H_i \prod_{j\neq i}Q_j.
  \end{equation}
  Let $\gamma$ be a good curve, and suppose that
  $\dist(\cdot,\NI(P;R))\rest\gamma$ attains its minimum on the
  component $\NC_i$. Then
  \begin{equation}
    \ord_\gamma H_i,\ord_\gamma Q_i \ge \ord_\gamma \dist(\cdot,\NI(P;R))
  \end{equation}
  and since each summand in~\eqref{eq:e-def} is a product
  containing either $H_i$ or $Q_i$,
  \begin{equation}
    \ord_\gamma E \ge \ord_\gamma \dist(\cdot,\NI(P;R)).
  \end{equation}
  Moreover, for each component $\NC_i$, all summands other
  than the $i$-th vanish identically on $\NC_i$, whereas the
  $i$-th summand does not. Therefore $E$ does not vanish identically
  on any component $\NC_i$, and the proposition is proved.
\end{proof}

\bibliographystyle{plain}
\bibliography{refs}

\end{document}